\documentclass[11pt]{amsart}
\usepackage{tikz}
\usetikzlibrary{positioning}
\usepackage{latexsym}
\usepackage{amssymb}
\usepackage{amsmath}
\usepackage{enumitem}
\usepackage{graphicx}
\setlist[enumerate,1]{font=\upshape,label=(\roman*)}
\usepackage{ifdraft}
\newtheorem{theorem}{Theorem}[section]
\newtheorem{lemma}[theorem]{Lemma}
\newtheorem{proposition}[theorem]{Proposition}
\newtheorem{corollary}[theorem]{Corollary}

\theoremstyle{definition}
\newtheorem{remark}[theorem]{Remark}
\newtheorem*{remark*}{Remark}

\theoremstyle{definition}
\newtheorem{example}[theorem]{Example}
\newtheorem{definition}[theorem]{Definition}

\newcommand{\oincl}{\subseteq_{\omega}}
\newcommand{\ocl}{\operatorname{cl}_{\omega}}
\newcommand{\ointer}{\operatorname{int}_{\omega}}
\newcommand{\oclint}{\operatorname{clint}_{\omega}}
\newcommand{\oclsq}{\operatorname{clsq}_{\omega}}
\newcommand{\osupp}{\supp_{\omega}}
\newcommand{\osq}{\operatorname{sq}_{\omega}}
\renewcommand{\oint}{\PackageError{latticeopsys}{you meant ointer not oint}{}}
\DeclareMathOperator{\RefHull}{Ref}
\DeclareMathOperator{\supp}{supp}
\DeclareMathOperator{\ran}{ran}

\def\loopy#1#2{\def#1##1{\def\next{#2{##1}#1}\ifx##1\relax\let\next\relax\fi\next}}
\def\dc#1{\expandafter\def\csname#1\endcsname{\mathcal{#1}}}
\def\db#1{\expandafter\def\csname b#1\endcsname{\mathbb{#1}}}
\def\df#1{\expandafter\def\csname f#1\endcsname{\mathfrak{#1}}}
\loopy{\makemathcals}{\dc}
\loopy{\makemathbbs}{\db}
\loopy{\makemathfraks}{\df}
\makemathbbs   CRNQTZKDS \relax
\makemathcals  QWERTYUIOPASDFGHJKLZXCVBNM\relax
\makemathfraks QWERTYUIOPASDFGHJKLZXCVBNM\relax

\newcommand{\cl}[1]{\mathcal{#1}}
\newcommand{\bb}[1]{\mathbb{#1}}

\usepackage[colorlinks=true,linkcolor=blue,citecolor=blue]{hyperref}
\hypersetup{final}

\begin{document}

\title{Lattices of strongly reflexive masa-bimodules}

\author[R. H. Levene]{Rupert H. Levene}
\address{School of Mathematics and Statistics, University College Dublin, Belfield, Dublin~4, Ireland}
\email{rupert.levene@ucd.ie}

\author[Y.-F. Lin]{Ying-Fen Lin}
\address{Mathematical Sciences Research Centre, Queen's University Belfast, Bel\-fast, BT7 1NN, United Kingdom}
\email{y.lin@qub.ac.uk}

\author[I. G. Todorov]{Ivan G. Todorov}
\address{School of Mathematical Sciences, University of Delaware, 501 Ewing Hall, Newark, DE 19716, USA}
\email{todorov@udel.edu}

\begin{abstract}
We characterise the density of the positive rank one subspace of a
masa-bimodule in terms of its support.
We prove that strongly reflexive masa-bimodules
form a Boolean lattice under naturally defined operations.
We examine the lattice-theoretic properties of the class of
strongly reflexive masa-bimodules that are also operator systems and some
natural subclasses thereof, and
provide a topological description of the lattice operations in the case the
masa-bimodules arise from closed subsets of a locally compact group.
\end{abstract}

\date{16 June 2026}

\maketitle

\section{Introduction}

Subspaces of the algebra $M_n$ of all complex $n$ by $n$ matrices, arising from
graphs, appear frequently in many areas of linear algebra and operator theory.
For instance, in the case of undirected graphs $G$ on $n$ vertices,
the operator system $\cl S(G)$ of all matrices in $M_n$ supported on the adjacency relation
of $G$ plays a distinctive role in non-commutative geometry \cite{cvs1, cvs2}
and provides a crucial passage from classical to quantum zero-error information theory
\cite{dsw}. The space $\cl S(G)$ is, in addition, closely related to the positive completion
problem for partially defined positive definite matrices; in fact, a fundamental
characterisation of undirected graphs that admit positive completions is provided in \cite{pps}
in terms of positive cone generation by positive rank one operators lying in $\cl S(G)$.

The operator systems (or, more generally, spaces, in the directed graph case) of the form $\cl S(G)$ are precisely the
operator subsystems (or, more generally, subspaces) of $M_n$ that are bimodules over the
diagonal subalgebra of $M_n$.
In the infinite dimensional case, a similar coordinate description of spaces that are
bimodules over maximal abelian selfadjoint algebras (masas, for short) is still available, but
not as trivial. Such a coordinatisation was first developed in
\cite{a} and later refined in \cite{eks} (see also \cite{st}):
assuming the separability of the underlying Hilbert space $H$, and realising $H = L^2(X,\mu)$
for some standard measure space $(X,\mu)$ in such a way that
the masa $\cl D\subseteq \cl B(H)$ can be identified with the set of
multiplication operators by functions in $L^{\infty}(X,\mu)$,
the support $\kappa$ of a $\cl D$-bimodule $\cl S$ is an
\emph{$\omega$-closed} subset of the Cartesian product $X\times X$, that is,
a subset $\kappa\subseteq X\times X$ that is (marginally equivalent to)
the complement of a countable union of measurable rectangles.
Such sets $\kappa$ are
an infinite dimensional, measurable, counterpart of the adjacency relation of a
(directed or undirected) graph in the finite case; furthermore, in case $\cl S$ is reflexive
in the sense of \cite{a, eks, ls, lt}, $\kappa$ completely determines $\cl S$. We thus denote this
by writing $\cl S = \cl S(\kappa)$.

A fundamental difference between the finite and the infinite dimensional case is
the general lack, in the latter, of rank one operators in the
masa-bimodules of the form $\cl S(\kappa)$ (see \cite{eks, llt}).
This led to a thorough investigation of the presence and the size of the subspace
of a masa-bimodule $\cl S\subseteq \cl B(H)$, generated by its rank one operators
(referred to as the \emph{rank one subspace} of $\cl S$).
The \emph{strongly reflexive} masa-bimodules in $\cl B(H)$ (that is, the masa-bimodules
that are the reflexive hull \cite{a} of a set of rank one operators) were shown in
\cite{eks} to have weakly dense (but not necessarily weak* dense) rank one subspace.
Moreover, the corresponding supports $\kappa$ were characterised in \cite{eks} as
being the $\omega$-closure of their $\omega$-interior (see Section \ref{s_lpd} for
the precise definitions).

In the case where $\cl S$ is an operator system, the support
is a \emph{positivity domain} -- a symmetric set that is the
$\omega$-closure of its $\omega$-interior,
containing the diagonal in $X\times X$.
The abundance of positive rank one operators in $\cl S$ in an appropriate sense
was shown in \cite{llt} to be equivalent
to the positive completion property for the underlying support $\kappa$.
Such sets are special classes of
\emph{graphons} in the sense of graph limit theory \cite{lovasz-book} and, in the
context of locally compact groups, arise naturally from symmetric sets containing the neutral
element, and relate to positive extensions of positive definite functions,
partially defined over the given group \cite{llt}.

Motivated by the contexts described in the preceding paragraphs, in this
paper we provide a characterisation of the masa-bimodules, generated by
their positive rank one subspace in terms of the underlying support (Theorem \ref{th_psr}).
We examine the lattice-theoretic properties of the class $\fM$ of
strongly reflexive masa-bimodules. We show that, under natural operations,
$\fM$ is a Boolean lattice. Interestingly, the join operation can be defined thanks to a
rectangle generation for unions of supports (Theorem \ref{th_oc}) which
implies, in particular, a weak version of the Baire Category Theorem for
\emph{$\omega$-dense sets} (Corollary \ref{c_unions}).
Our approach relies on a Kuratowski-type regularisation
operation for measurable subsets of $X\times X$ that complements the study of $\omega$-closed sets as found in a series of articles (see \cite{eks, ks, llt, st} and the
references therein).

We show that the strongly reflexive operator systems that are masa-bimodules
form a sublattice of $\fM$ (Theorem \ref{th_com_rest}).
It was demonstrated in \cite{llt} that there are strongly reflexive operator systems with no positive rank one operators.
This is our motivation to introduce
\emph{positive-full} and \emph{positive-null} operator systems,
which lie at the extremal ends of the spectrum when it comes
to positive completion properties.
The positive-full operator systems admit a
lattice structure as well; while the positive-null operator systems are naturally not closed
under the formation of joins, they are, surprisingly, closed under complementation (see
Theorem \ref{th_com_rest}).

In the last section of the paper (Section \ref{s_mbtt}) we focus on strongly reflexive masa-bimodules of Toeplitz type.
Such masa-bimodules have supports of the form $E^* := \{(s,t)\in G\times G : ts^{-1}\in E\}$,
where $G$ is a locally compact group and $E\subseteq G$ is a closed symmetric set containing the neutral element of $G$, and
have received substantial attention in the literature
(see e.g. \cite{akt, lt, st}) as an invariant counterpart of general masa-bimodules.
We characterise the containment of a masa-bimodule of Toeplitz type in any of the
classes we consider, and describe explicitly the lattice operations in topological terms.

\smallskip

\noindent {\bf Acknowledgements. }
IT was supported by NSF grants
CCF-2115071 and DMS-2154459.
This material is based upon work supported by the
Swedish Research Council under grant no.\ 2021-06594 while the
first and the third author were in residence at Institut Mittag-Leffler in Djursholm, Sweden, during the 2026 program on Operator Algebras and Quantum Information.

\section{Properties of \texorpdfstring{$\omega$}{ω}-topology}\label{s_lpd}

In this section we recall the basic definitions and properties of
what is known as the \emph{$\omega$-topology} on product measure spaces, following \cite{eks, stt, llt}.  The main new results are Theorem~\ref{th_oc}, which describes how a natural regularisation operation interacts with finite unions of $\omega$-closed sets, and Corollary~\ref{c_unions}, which establishes a finite intersection property for $\omega$-dense sets reminiscent of the Baire Category Theorem.

We begin by recalling some relevant background from \cite{eks}.
Let $(X,\mu)$ be a standard $\sigma$-finite measure space.
We equip $X\times X$ with the product measure $\mu\times\mu$.
A measurable subset $E\subseteq X\times X$ is said to be \emph{marginally null}
if there exists a null set $M\subseteq X$ such that
$E\subseteq (M\times X) \cup (X\times M)$.
If $\kappa_1$ and $\kappa_2$ are measurable subsets of $X\times X$,
we say that $\kappa_1$ is \emph{marginally contained in} $\kappa_2$
(and write $\kappa_1\oincl \kappa_2$) if $\kappa_1\setminus \kappa_2$ is marginally null.
We call the relation $\oincl$ \emph{$\omega$-inclusion}.
We say that $\kappa_1$ and $\kappa_2$ are \emph{marginally equivalent} (and write $\kappa_1\cong \kappa_2$)
if $\kappa_1\oincl\kappa_2$ and $\kappa_2\oincl\kappa_1$
(or, equivalently, if the symmetric difference of $\kappa_1$ and $\kappa_2$ is marginally null).

A subset of $X\times X$ of the form $\alpha\times\beta$ (resp.~$\alpha\times\alpha$),
where $\alpha\subseteq X$ and $\beta\subseteq X$ are measurable sets,
will be called a
\emph{rectangle} (resp. \emph{square}) in $X\times X$.
A rectangle or a square is said to be \emph{non-trivial} if it is not marginally null.
A subset $\Omega\subseteq X\times X$
is called \emph{$\omega$-open} (resp. \emph{square~$\omega$-open})
if there exists a countable family
$\{R_i\}_{i\in \bb{N}}$ of rectangles in $X\times X$ (resp. squares in $X\times X$)
such that $\Omega\cong \bigcup_{i\in \bb{N}} R_i$.
A subset $\kappa\subseteq X\times X$ is called \emph{$\omega$-closed} if its complement
$\kappa^c$ is $\omega$-open. Clearly, the class of $\omega$-open (resp.~$\omega$-closed) sets is closed under
the formation of countable unions and finite intersections (resp.~countable intersections and finite unions).

It was shown in \cite{stt} that if $\E$ is an arbitrary family of $\omega$-open subsets of $X\times X$
then there exists a smallest, up to marginal equivalence, $\omega$-open set
$\Omega=\bigcup_{\omega}\E$, called the \emph{$\omega$-union} of $\E$, such that
$$E\in \E \ \Longrightarrow \ E\oincl \Omega.$$
Moreover, $\Omega$ is marginally equivalent to the union of countably many elements of~$\E$.
Given a measurable set $\kappa$, let
\[\ointer(\kappa) :=
\cup_{\omega}\{R : R \, \text{ is a rectangle with } R \subseteq_{\omega} \kappa\}\]
be its \emph{$\omega$-interior}, and
$\ocl(\kappa) := \ointer(\kappa^c)^c$ be its \emph{$\omega$-closure}; as we point out below, this operation satisfies the standard Kuratowski closure axioms with respect to marginal equivalence and marginal inclusions.
A measurable set~$\kappa\subseteq X\times X$ is said to be \emph{$\omega$-dense}
if~$\ocl(\kappa)=X\times X$.
Similarly, we let
\[\osq(\kappa) :=
\cup_{\omega}\{R : R \, \text{ is a square with } R \subseteq_{\omega} \kappa\}\]
be the \emph{square $\omega$-interior} of $\kappa$.

\begin{lemma}\label{lem_ointer}
  Let~$\kappa,\kappa_1,\kappa_2\subseteq X\times X$ be measurable
  sets. Then $\ointer(\kappa)$ is the largest, up to marginal
  equivalence, $\omega$-open set marginally contained in~$\kappa$.
  Hence, $\kappa$ is $\omega$-open if and only if $\kappa\cong \ointer(\kappa)$.
  Moreover,
  \begin{enumerate}
    \item $\ointer(X\times X)\cong X\times X$;
    \item $\ointer(\ointer(\kappa))\cong\ointer(\kappa)$;
    \item $\kappa_1\subseteq_\omega \kappa_2\implies \ointer(\kappa_1)\subseteq_\omega \ointer(\kappa_2)$, and
    \item $\ointer(\kappa_1\cap \kappa_2)\cong\ointer(\kappa_1)\cap
      \ointer(\kappa_2)$.
  \end{enumerate}
The same assertions hold, mutatis mutandis, for the square $\omega$-interior operation $\osq$ in place of $\ointer$.
\end{lemma}
\begin{proof}
  The first assertion and parts (i)--(iii) are straightforward. We show (iv).
Clearly, the set $\Omega := \ointer(\kappa_1)\cap \ointer(\kappa_2)$ is $\omega$-open
  and marginally contained in $\kappa_1\cap \kappa_2$. Thus,
  $\Omega = \ointer(\Omega)\subseteq_\omega \ointer(\kappa_1\cap
  \kappa_2)$. On the other hand, if~$R$ is a rectangle with
  $R\subseteq \kappa_1\cap \kappa_2$, then $R\subseteq \kappa_1$ and
  $R\subseteq \kappa_2$, so $R\subseteq_\omega \Omega$. Hence
  $\ointer(\kappa_1\cap \kappa_2)\subseteq_\omega \Omega$ and (iv) follows.

Arguing with squares in place of rectangles, we obtain the corresponding properties of $\osq$;
the only point that needs clarification is that the set
$\Omega' := \osq(\kappa_1)\cap \osq(\kappa_2)$ is square $\omega$-open.
To this end, notice that if, up to marginal equivalence,
$\osq(\kappa_1) = \cup_{i\in \bb{N}} \alpha_i\times\alpha_i$ and
$\osq(\kappa_2) = \cup_{j\in \bb{N}} \beta_j\times\beta_j$ then, up to marginal
equivalence,
\[\Omega' = \cup_{i,j\in \bb{N}} (\alpha_i\cap\beta_j)\times (\alpha_i\cap\beta_j).\qedhere\]
\end{proof}

Taking complements, Lemma \ref{lem_ointer} implies the following
version of Kuratowski's closure axioms for $\omega$-topology.

\begin{corollary}\label{lem_oc}
  Let~$\kappa,\kappa_1,\kappa_2\subseteq X\times X$ be measurable
  sets. Then $\ocl(\kappa)$ is the smallest, up to marginal
  equivalence, $\omega$-closed set marginally containing
  $\kappa$. Hence, $\kappa$ is $\omega$-closed if and only if $\kappa\cong \ocl(\kappa)$.
  Moreover,
  \begin{enumerate}
  \item $\ocl(\emptyset) \cong \emptyset$;
  \item $\ocl(\ocl(\kappa)) \cong \ocl(\kappa)$;
  \item $\kappa_1\subseteq_\omega \kappa_2\implies \ocl(\kappa_1)\subseteq_\omega\ocl(\kappa_2)$, and
  \item $\ocl(\kappa_1\cup \kappa_2) \cong \ocl(\kappa_1) \cup
    \ocl(\kappa_2)$.
  \end{enumerate}
\end{corollary}

We consider two regularisation operations, $\oclint$ and $\oclsq$, defined for a measurable subset $\kappa\subseteq X\times X$ as follows:
\[\oclint(\kappa) = \ocl(\ointer(\kappa)) \ \mbox{ and } \  \oclsq(\kappa)=\ocl(\osq(\kappa)).\]
The preceding considerations immediately imply that these operations both preserve marginal inclusions.
The next definition was given in \cite{llt}.

\begin{definition}\label{d_genrs}
A measurable subset~$\kappa\subseteq X\times X$ is said to be:
\begin{enumerate}
\item\emph{generated by rectangles} if $\kappa\cong\oclint(\kappa)$;
\item\emph{generated by squares} if $\kappa\cong\oclsq(\kappa)$.
\end{enumerate}
\end{definition}

\begin{remark}\label{r_gsq}
A measurable subset~$\kappa\subseteq X\times X$ is generated by rectangles (resp. squares) if and only if
there exists an $\omega$-open set (resp. square $\omega$-open set) $\Omega$ such that $\kappa \cong \ocl(\Omega)$. In particular, $\oclint(\kappa)$ (resp. $\oclsq(\kappa)$) is generated by rectangles (resp. squares).
\end{remark}

\begin{proof}
We give the argument for squares; the argument for rectangles is identical.
If $\Omega$ is a square $\omega$-open set such that $\kappa \cong \ocl(\Omega)$,
then we have \[\osq(\Omega)\cong\Omega \oincl \ocl(\Omega) \cong \kappa.\]
Since $\ocl$ and $\osq$ both preserve marginal inclusions (see Lemma \ref{lem_ointer}),
it follows that
\[ \kappa \cong \ocl(\Omega) \cong \oclsq(\Omega) \subseteq_\omega \oclsq(\kappa) \subseteq_\omega \ocl(\kappa) \cong \kappa. \]
Thus $\kappa \cong \oclsq(\kappa)$, meaning $\kappa$ is generated by squares.

Conversely, if $\kappa$ is generated by squares, then $\kappa \cong\ocl(\Omega)$ where $\Omega=\osq(\kappa)$, which is square $\omega$-open.
\end{proof}

\begin{theorem}\label{th_oc}
Let $\kappa_1$ and $\kappa_2$ be $\omega$-closed subsets of $X\times X$. Then
\begin{equation}\label{eq_3cl}
\oclint(\kappa_1\cup \kappa_2) \cong \ocl(\ointer(\kappa_1) \cup \ointer(\kappa_2)) \cong
\oclint(\kappa_1) \cup \oclint(\kappa_2).
\end{equation}
\end{theorem}

\begin{proof}
The second equivalence in (\ref{eq_3cl}) follows from Corollary \ref{lem_oc}\,(iv).
For the first equivalence, let $\Omega_i = \ointer(\kappa_i)$, $i = 1,2$. Write $\Omega=\Omega_1\cup \Omega_2$ and $\kappa=\kappa_1\cup \kappa_2$; we wish to show that $\ocl(\Omega)\cong \oclint(\kappa)$.
Since $\Omega \oincl \kappa$ and $\Omega$ is an $\omega$-open set,
we have $\Omega \oincl \ointer(\kappa)$; hence, $\ocl(\Omega) \oincl \oclint(\kappa)$.

To prove the reverse inclusion, it suffices to show the following

\medskip

\noindent {\it Claim. } If $\lambda_1$ and $\lambda_2$ are $\omega$-closed subsets
of a non-trivial rectangle $R\subseteq X\times X$ such that
$\lambda_1\cup\lambda_2 = R$, then at least one of
$\lambda_1$ and $\lambda_2$ has a non-empty $\omega$-interior.

\medskip

Indeed, suppose that the Claim is true, and assume,
towards a contradiction, that $\ointer(\kappa)\not\oincl\ocl(\Omega)$, i.e., that
the set $U := \ointer(\kappa) \setminus \ocl(\Omega)$
is not marginally null.
Since $U$ is $\omega$-open, there exists a non-trivial rectangle $R \subseteq_\omega U$.
For $i=1,2$, we have $R\cap\Omega_i\oincl R\cap \ocl(\Omega)\cong\emptyset$, so $R\cap\Omega_i\cong\emptyset$.
On the other hand, $R\subseteq \kappa$. Letting $\lambda_i = R\cap \kappa_i$, $i = 1,2$,
we have that $R = \lambda_1\cup \lambda_2$ and, by the Claim, we may assume, without
loss of generality, that there exists a rectangle $R'$ with $\emptyset\not\cong R'\subseteq_{\omega} R$, such that $R'\subseteq \lambda_1$.
It follows that $R'\subseteq_{\omega} \Omega_1$, contradicting the fact that
$R'\cap\Omega_1\cong\emptyset$.
We conclude that $\ointer(\kappa)\oincl\ocl(\Omega)$, which implies that $\oclint(\kappa)\oincl\ocl(\Omega)$.

In order to prove the Claim, consider the $\omega$-open sets $W_i = R\setminus \lambda_i$,
$i = 1,2$, and note that the assumption of the Claim reads
$W_1\cap W_2 \cong\emptyset$. Suppose,
towards a contradiction, that $W_1$ and $W_2$ are both $\omega$-dense in $R$, i.e., that $\ocl(W_i)=R$, $i=1,2$.
Then $W_1$ is not marginally null and $\omega$-open, so there exists a non-marginally null
rectangle $S$ with $S\subseteq_{\omega} W_1$. Now $W_1\cap W_2\cong\emptyset$ implies that $S\cap W_2 \cong\emptyset$, so $W_2\oincl S^c$. Since $S^c$ is $\omega$-closed, this means that $\ocl (W_2) \oincl S^c$,
a contradiction with the fact that $W_2$ is $\omega$-dense in $R$.
\end{proof}

\begin{corollary}\label{c_gru}
  If $\kappa_1$ and $\kappa_2$ are subsets of $X\times X$ that are generated by rectangles, then
  $\kappa_1\cup \kappa_2$ is generated by rectangles.
\end{corollary}
\begin{proof}
By Theorem \ref{th_oc},
\begin{equation*}
\kappa_1\cup \kappa_2
 \cong
  \oclint(\kappa_1) \cup \oclint(\kappa_2)
  \cong
\oclint(\kappa_1\cup \kappa_2).\qedhere
\end{equation*}
\end{proof}

Theorem \ref{th_oc} has the following consequence.

\begin{corollary}\label{c_unions}
The intersection of a finite family of $\omega$-open and $\omega$-dense subsets of $X\times X$ is $\omega$-dense.
\end{corollary}
\begin{proof}
It suffices to prove the statement for two $\omega$-open and $\omega$-dense sets $\Omega_1$ and $\Omega_2$.
Since $\ointer(\Omega_i^c) \cong \emptyset$, $i = 1,2$,
Theorem~\ref{th_oc} yields $\oclint(\Omega_1^c \cup\Omega_2^c) \cong \emptyset$, so we have $\ointer(\Omega_1^c\cup \Omega_2^c)\cong \emptyset$,
and hence $\ocl(\Omega_1\cap \Omega_2) \cong X\times X$.
\end{proof}

\begin{remark*}
Corollary \ref{c_unions} can be considered as a weak version of
the Baire Category Theorem for $\omega$-topology.
We do not know if its statement remains true for countably many $\omega$-open and $\omega$-dense subsets of $X\times X$.
\end{remark*}

\section{Reflexivity and positivity domains}\label{s_rposdom}

In this section, we set up the measure-theoretic and operator-algebraic framework for the remainder of the paper. After fixing standard notation and recalling the correspondence between $\omega$-closed sets and reflexive masa-bimodules, we refine a fundamental reflexivity theorem of Erdos, Katavolos, and Shulman \cite{eks}. We then adapt these spatial concepts to the setting of operator systems, introducing the notion of a positive-full subspace, and provide an analogous spatial characterisation for this strengthened form of reflexivity.

Let $(X,\mu)$ be a standard measure space,
and set $H = L^2(X,\mu)$.
We denote by $\B(H)$ the space of all bounded linear operators on $H$.
We refer to the weak operator topology on $\cl B(H)$ as the WOT.
The norm closure of a subset $S$ of a normed vector space, such as $H$ or $\cl B(H)$, will be written as $\overline{S}$,
whereas the closure of $S\subseteq \B(H)$ in the WOT will be denoted by $\overline{S}^w$.
For $\xi,\eta\in H$, let $\eta\xi^*$ be the (rank one) operator given by $(\eta\xi^*)(\zeta) = \langle \zeta,\xi\rangle\eta$, and
write
$$\R(H) = \{\eta\xi^* : 0\ne \xi,\eta\in H\}.$$
We let $\B(H)^+$ be the cone of all positive operators in $\B(H)$, and
set $\R(H)^+ = \B(H)^+ \cap \R(H)$.
We denote by $[\E]$ the linear span of a subset $\E$ of a linear space.
For $a\in L^{\infty}(X,\mu)$, we let $M_a\in \B(H)$ be the multiplication operator given by $M_a\xi = a\xi$, $\xi \in H$,
and write
$$\D = \{M_a : a\in L^{\infty}(X,\mu)\}.$$
Note that $\D$ is a maximal abelian selfadjoint algebra
(masa, for short) in $\B(H)$. For a measurable subset $\alpha\subseteq X$, let $P(\alpha)$ be the
operator of multiplication by the characteristic function of $\alpha$.
Given $\U\subseteq \B(H)$, let $[\U]_{\D}$ be the linear span of
$\D\U\D:=\{D_1TD_2:D_1,D_2\in \D,\,T\in \U\}$.
We call $\U$ a \emph{$\D$-bimodule}, or just a \emph{masa-bimodule}, if $[\U]_{\D} = \U$.

Let $\kappa\subseteq X\times X$ be a measurable subset.
An operator $T\in \B(H)$ is said to be \emph{supported by} $\kappa$ \cite{a,eks} if
$$\left(\alpha\times\beta\right) \cap \kappa \cong \emptyset \implies P(\beta)TP(\alpha) = 0,$$
for all measurable subsets $\alpha, \beta\subseteq X$. Set
$$\S(\kappa) = \{T\in \B(H) : T \text{ is supported by } \kappa\}.$$
It is straightforward to check that $\S(\kappa)$ is a WOT-closed masa-bimodule.

The masa-bimodules of the form $\S(\kappa)$ are closely related to the concept of reflexivity.
Recall that, given a subset $\U\subseteq \B(H)$, its \emph{reflexive cover}
\cite{ls} is the subspace
$$\RefHull(\U) = \{T\in \B(H) : T\xi\in \overline{[\U \xi]}, \text{ for every } \xi\in H\}.$$
It is clear that $\U\subseteq \RefHull(\U) = \RefHull([\U])$, and that $\RefHull(\U)$ is closed in the WOT. Moreover, $\RefHull(\U)$ is a masa-bimodule whenever $\U$ is a masa-bimodule.
A subspace $\U\subseteq \B(H)$ is called \emph{reflexive} if $\U = \RefHull(\U)$,
and it is called \emph{strongly reflexive} if there exists a set $\R\subseteq \R(H)$
such that $\U = \RefHull(\R)$.

For an arbitrary subset $\U \subseteq \B(H)$, its \emph{$\omega$-support}, denoted $\osupp(\U)$, is defined as the unique (up to marginal equivalence) $\omega$-closed set $\kappa$ such that $\RefHull([\U]_\D) = \S(\kappa)$ \cite{eks}. It is straightforward to see from this definition that $\osupp(\U)$ is also the minimal,
with respect to $\omega$-inclusion, $\omega$-closed set $\lambda$ such that $\U \subseteq \S(\lambda)$. Furthermore,
$\U$ is a reflexive masa-bimodule if and only if $\U = \S(\osupp(\U))$. Hence, every reflexive
masa-bimodule $\U$ is of the form $\U=\S(\kappa)$ for some $\omega$-closed subset $\kappa$ of $X\times X$, namely $\kappa=\osupp(\U)$, and $\U\mapsto \osupp(\U)$ is a bijective correspondence between the reflexive masa-bimodules and the $\omega$-closed subsets of $X\times X$, modulo marginal equivalence (see \cite{eks}).

We recall a fundamental characterisation of strong reflexivity, due to Erdos, Katavolos and Shulman \cite{eks}. We say that a WOT-closed subspace $\U\subseteq \B(H)$ is \emph{generated by its rank one operators} if $\U=\overline{[\R(H)\cap \U]}^w$.

\begin{theorem}[{\cite[Theorem~6.11]{eks}}]\label{th_sr}
Let $\kappa\subseteq X\times X$ be an $\omega$-closed set. The following are equivalent:
\begin{enumerate}
\item the masa-bimodule $\S(\kappa)$ is strongly reflexive;
\item $\kappa$ is generated by rectangles;
\item $\S(\kappa)$ is generated by its rank one operators.
\end{enumerate}
Hence, $\U\mapsto \osupp(\U)$ is a bijective correspondence between the strongly reflexive
masa-bimodules and the subsets of $X\times X$ which are generated by rectangles, modulo marginal equivalence.
\end{theorem}

For $\kappa\subseteq X\times X$, call the set $\kappa^t=\{(y,x):(x,y)\in \kappa\}$ the \emph{transpose} of $\kappa$. Clearly, $\ointer$, $\osq$ and $\ocl$ commute with the transpose.
Let $\Delta = \{(x,x) : x\in X\}$ be the \emph{diagonal} in $X\times X$.
A measurable subset~$\kappa\subseteq X\times X$ is said to be a \emph{positivity domain} \cite{llt} if
it is generated by rectangles, it is symmetric (that is, $\kappa\cong \kappa^t$), and
$\Delta\oincl\kappa$.
Recall that an \emph{operator system} is a subspace $\S\subseteq \B(H)$ such that
$I\in \S$ and $T^*\in \S$ for every $T\in \S$.
An operator system $\S\subseteq \B(H)$ that is also a $\D$-bimodule will be called
an \emph{operator $\D$-system}. The following fact was proved in \cite{llt}:

\begin{proposition}[{\cite[Proposition~3.3]{llt}}]\label{p_llt}
  If $\kappa\subseteq X\times X$ is generated by rectangles,  then
 $\S(\kappa)$ is an operator system if and only if $\kappa$ is a positivity domain.
\end{proposition}

In the next definition, we introduce a version of reflexivity that is stronger than
strong reflexivity and useful in the context of operator systems.
It is motivated by the observation in \cite{llt} that there exist strongly reflexive operator systems that do not contain positive rank one operators, and the
link established therein between the abundance of positive rank one operators
in a masa-bimodule and the positive completion properties of its support.

\begin{definition}\label{d_psr}
A subspace $\U\subseteq \B(H)$ is called
\begin{enumerate}
\item \emph{positive-full} if
there exists a set $\R\subseteq \B(H)$ of positive rank one operators such that $\U = \RefHull(\R)$;
\item \emph{positive-null} if $\U$ does not contain a positive rank one operator.
\end{enumerate}
\end{definition}

Of course, any positive-full subspace is strongly reflexive.

\begin{remark*}
  Let $\kappa\subseteq X\times X$ be a positivity domain.
  It is straightforward to see that $\S(\kappa)$
  contains a positive rank one operator if and only if there exists
  a non-null measurable subset $\alpha\subseteq X$ such that $\alpha\times\alpha\oincl \kappa$.
  Thus, $\S(\kappa)$ is positive-null if and only if $\osq(\kappa)\cong \emptyset$.
  For a concrete construction of a positive-null operator $\D$-system over the circle group, we refer the reader to \cite[Corollary~5.3]{llt}.
\end{remark*}

In Theorem \ref{th_psr} below, we prove an operator system version of Theorem \ref{th_sr}.
We will need the following fact:

\begin{lemma}
  \label{lem:squares}
  If $\R\subseteq \R(H)^+$ then $\osupp(\R)$ is generated by squares.
\end{lemma}

\begin{proof}
The proof proceeds similarly to that of \cite[Lemma~5.1]{eks};
we outline it for the convenience of the reader.
Fix a countable dense subset $\{\xi_i\xi_i^*\}_{i\in \bb{N}}$ of $\cl R$, and
let $\alpha_i$ be the essential support of $\xi_i$, $i\in \bb{N}$.
Let
$\kappa =
\operatorname{cl}_{\omega}
\left(
\bigcup_{i=1}^{\infty}
(\alpha_i\times \alpha_i)
\right)$; by its definition, $\kappa$ is generated by squares.
It is now straightforward to check that $\kappa = {\rm supp}_{\omega}(\cl R)$.
\end{proof}

\begin{theorem}\label{th_psr}
Let $\kappa\subseteq X\times X$ be an $\omega$-closed set. The following are equivalent:
\begin{enumerate}
\item $\S(\kappa)$ is positive-full;
\item $\kappa$ is generated by squares;
\item $\S(\kappa) = \overline{[\R(H)^+\cap \S(\kappa)]_{\D}}^{w}$.
\end{enumerate}
\end{theorem}

\begin{proof}
(i)$\Rightarrow$(ii). Let $\R\subseteq \R(H)^+$ be such that $\S(\kappa) = \RefHull(\R)$. Then $\kappa \cong \osupp(\S(\kappa)) \cong \osupp(\R)$. By Lemma~\ref{lem:squares}, $\kappa$ is generated by squares.\medskip

(ii)$\Rightarrow$(iii) and (i). By assumption, there exist measurable subsets $\alpha_i \subseteq X$, $i\in \bb{N}$, such that $\kappa \cong \ocl\left(\bigcup_{i\in \bb{N}} \alpha_i\times\alpha_i\right)$. Let
$$\R' = \bigcup_{i\in \bb{N}} \{\eta\xi^* : \xi, \eta \in \ran P(\alpha_i)\} \quad \text{and} \quad \R = \bigcup_{i\in \bb{N}} \{\xi\xi^* : \xi \in \ran P(\alpha_i)\}.$$
By the proof of \cite[Theorem 5.2]{eks}, $\S(\kappa) = \RefHull(\R')$. Since $\S(\kappa)$ is a strongly reflexive masa-bimodule, \cite[Theorem 6.11]{eks} implies that $\S(\kappa) = \overline{[\R']}^{w}$. By polarisation, $\R'\subseteq [\R]$. Since $\R \subseteq \R(H)^+\cap \S(\kappa)$, we obtain
$$\S(\kappa) = \overline{[\R']}^{w} \subseteq \overline{[\R]_{\D}}^{w} \subseteq \overline{[\R(H)^+\cap \S(\kappa)]_{\D}}^{w} \subseteq \overline{[\S(\kappa)]_\D}^w \subseteq \S(\kappa).$$
This establishes (iii).
Moreover, we have
$$\S(\kappa) = \RefHull(\R') \subseteq \RefHull([\R]) = \RefHull(\R) \subseteq \S(\kappa),$$
showing that $\S(\kappa) = \RefHull(\R)$ and hence that $\S(\kappa)$ is positive-full,
that is, (i) holds.\medskip

(iii)$\Rightarrow$(ii). Let $\R = \R(H)^+\cap \S(\kappa)$; by assumption, $\S(\kappa) = \overline{[\R]_{\D}}^{w}$.
Every $T \in \R$ is a positive rank-one operator, so its $\omega$-support is a square and hence is $\omega$-open. Let $\E = \{\osupp(T) : T \in \R\}$. By Remark~\ref{r_gsq}, $\kappa_0:= \ocl\left( \bigcup_\omega \E \right)$ is generated by squares, so it suffices to show that $\kappa\cong\kappa_0$.

By construction, we have $\R\subseteq \S(\kappa_0)$, so $\S(\kappa)=\overline{[\R]_\D}^w\subseteq \S(\kappa_0)$; hence, $\kappa\oincl\kappa_0$.
Conversely, since $\R \subseteq \S(\kappa)$, we have $E \oincl \kappa$ for all $E \in \E$. Each such $E$ is $\omega$-open, so $E\oincl\ointer(\kappa)$. Hence, $\bigcup_\omega \E\oincl\ointer(\kappa)$, and taking $\omega$-closures yields $\kappa_0\oincl \oclint\kappa\oincl\kappa$, so $\kappa\cong\kappa_0$, as required.
\end{proof}

We will need the following more precise version of the equivalence of statements (ii) and~(iii) in Theorem \ref{th_sr}, which will be used later.

\begin{proposition}\label{l_goi}
Let $\kappa\subseteq X\times X$ be an $\omega$-closed set.
The following assertions hold:
\begin{enumerate}
\item[(i)] $\overline{[\R(H)\cap \S(\kappa)]}^{w} = \S(\oclint(\kappa))$;

\item[(ii)] $\overline{[\R(H)^+\cap \S(\kappa)]_\D}^{w} = \S(\oclsq(\kappa))$.
\end{enumerate}
\end{proposition}

\begin{proof}
(i)
Since $\kappa$ is $\omega$-closed, we have $\oclint(\kappa)\oincl\kappa$. The set $\oclint(\kappa)$ is generated by rectangles, so by Theorem \ref{th_sr}, we have
\[\S(\oclint(\kappa)) = \overline{[\R(H)\cap \S(\oclint(\kappa))]}^{w}\subseteq\overline{[\R(H)\cap \S(\kappa)]}^{w}.\]

On the other hand, any $\eta\xi^*\in \R(H)\cap \S(\kappa)$ has support $R = \supp(\xi)\times\supp(\eta)\oincl\kappa$, so
$R \oincl \ointer(\kappa) \oincl \oclint(\kappa)$.
Hence,  $\eta\xi^* \in \S(\oclint(\kappa))$, so $\R(H)\cap \S(\kappa) \subseteq \S(\oclint(\kappa))$,
and thus $\overline{[\R(H)\cap \S(\kappa)]}^{w} \subseteq \S(\oclint(\kappa))$.

(ii) The argument is similar to that in (i), but we include it for completeness.
Since $\kappa$ is $\omega$-closed, we have $\oclsq(\kappa)\oincl\kappa$. The set $\oclsq(\kappa)$ is generated by squares, so by Theorem~\ref{th_psr} we have
\[\S(\oclsq(\kappa)) = \overline{[\R(H)^+\cap \S(\oclsq(\kappa))]_\D}^{w}\subseteq\overline{[\R(H)^+\cap \S(\kappa)]_\D}^{w}.\]

On the other hand, any $\xi\xi^*\in \R(H)^+\cap \S(\kappa)$ has support $Q = \supp(\xi)\times\supp(\xi)\oincl\kappa$, so
$Q \oincl \osq(\kappa) \oincl \oclsq(\kappa)$.
Hence,  $\xi\xi^* \in \S(\oclsq(\kappa))$, so $\R(H)^+\cap \S(\kappa) \subseteq \S(\oclsq(\kappa))$ and the conclusion follows.
\end{proof}

\section{Lattice structures}\label{s_lmb}

In this section, we examine the lattice-theoretic properties of various classes of strongly reflexive masa-bimodular subspaces and operator systems under natural lattice operations defined using rank one operators.
Fix $H=L^2(X,\mu)$ and $\D=L^\infty(X,\mu)$ as in the previous section, and let
\[ \fM = \{\S(\kappa):\text{$\kappa\subseteq X\times X$ and $\kappa$ is generated by rectangles}\}.\]

For $\S_1,\S_2\in \fM$, define
\begin{equation}\label{eq_defofvee}
\S_1 \vee \S_2 = \overline{\S_1 + \S_2}^w,
\quad \S_1 \wedge \S_2 = \overline{[\R(H)\cap \S_1\cap \S_2]}^w
\end{equation}
and
$$\S_1 \wedge_+ \S_2 = \overline{[\R(H)^+\cap \S_1\cap \S_2]_{\D}}^w.$$
It is clear that $\S_1 \vee \S_2$, $\S_1 \wedge \S_2$ and $\S_1 \wedge_+ \S_2$ are
WOT-closed masa-bimodules.

\begin{lemma}\label{l_lat}
  Let $\S_i\in \fM$, and let $\kappa_i\subseteq X\times X$ be
  $\omega$-closed sets, generated by rectangles, such that $\S_i=\S(\kappa_i)$, $i=1,2$. Then
  \begin{enumerate}
  \item $\S_1\vee \S_2 \in \fM$ and $\S_1\vee \S_2=\S(\kappa_1\cup \kappa_2)$;
  \item $\S_1\cap \S_2=\S(\kappa_1\cap \kappa_2)$;
  \item $\S_1\wedge \S_2 \in \fM$ and
  $\S_1\wedge \S_2=\S(\oclint(\kappa_1\cap \kappa_2))$.
  \end{enumerate}
  Moreover, $\S_1\vee\S_2$ (resp.~$\S_1\wedge\S_2$) is the least upper bound (resp.~the greatest lower bound) of $\S_1$ and $\S_2$ in $\fM$ with respect to inclusion.
\end{lemma}
\begin{proof}
(i)
Since $\S(\kappa_1 \cup \kappa_2)$ is a WOT-closed masa-bimodule containing both $\S_1$ and $\S_2$, we have $\S_1\vee\S_2 \subseteq \S(\kappa_1 \cup\kappa_2)$, hence $\osupp(\S_1\vee\S_2) \oincl \kappa_1\cup\kappa_2$.

Conversely, if $\alpha\times\beta$ is a rectangle with $(\alpha\times\beta) \cap \osupp(\S_1\vee\S_2) \cong \emptyset$, then $P(\beta)(\S_1 + \S_2)P(\alpha) = \{0\}$. Thus, $P(\beta)\S_i P(\alpha) = \{0\}$, implying that $(\alpha\times\beta) \cap \kappa_i \cong \emptyset$ for $i=1,2$. Consequently, $(\alpha\times\beta) \cap (\kappa_1\cup\kappa_2) \cong \emptyset$, which establishes that $\kappa_1\cup\kappa_2 \oincl \osupp(\S_1\vee\S_2)$.
Thus,
\begin{equation}\label{eq_k12s}
\osupp(\S_1\vee\S_2) \cong \kappa_1\cup\kappa_2.
\end{equation}
Using Theorem~\ref{th_sr}, we have that $\S_1$ and $\S_2$ are generated by their rank one operators, so the WOT-closed masa-bimodule $\S_1\vee\S_2=\overline{\S_1 + \S_2}^w$ is also generated by its rank one operators, and hence is strongly reflexive and completely determined by its $\omega$-support. Therefore, statement (i) follows directly from (\ref{eq_k12s}).
It is clear that $\S_1 \vee \S_2$ is the least upper bound of $\S_1$ and $\S_2$ in $\fM$ with respect to inclusion.

(ii) is easy to verify; the argument is omitted.

(iii)
Let $\kappa=\oclint(\kappa_1\cap\kappa_2)$. Then $\S(\kappa)\in \fM$ by Theorem~\ref{th_sr}. Using (ii) and Proposition \ref{l_goi}, we have $\S_1\wedge \S_2=\S(\kappa)$, so (iii) holds.
To see that $\S_1\wedge\S_2$ is the greatest lower bound of $\S_1$ and $\S_2$ with respect to inclusion, suppose that $\S\in \fM$ is such that
$\S\subseteq \S_1 \cap \S_2$. Letting $\lambda = \osupp(\S)$, we then have that
$\lambda \oincl \kappa_1\cap \kappa_2$, and hence
$$\lambda \cong \oclint(\lambda) \oincl \oclint(\kappa_1\cap \kappa_2)=\kappa,$$
so $\S=\S(\lambda)\subseteq \S(\kappa)= \S_1\wedge \S_2$, as required.
\end{proof}

If $\kappa$ is an $\omega$-closed set, we define its \emph{complementary set} $\kappa'$
by letting $\kappa' = \ocl(\kappa^c)$. Recall that if $\lambda\subseteq X\times X$ is measurable,
then $\ocl(\lambda)=\ointer(\lambda^c)^c$. Taking $\lambda=\kappa^c$ we obtain that
\begin{equation}\label{eq_com'}
\kappa' = \ocl(\kappa^c)=\ocl(\lambda)=\ointer(\lambda^c)^c=\ointer(\kappa)^c.
\end{equation}

Let us write $\mathbf{1} = \cl B(H)$ and $\mathbf{0}=\{0\}$,
and, when $\kappa\subseteq X\times X$ is generated by rectangles, set
$\S(\kappa)' := \S(\kappa')$.

\begin{theorem}\label{th_com}
$(\fM,\vee,\wedge,\mathbf{0},\mathbf{1},{}^\prime)$ is a Boolean algebra.
\end{theorem}
\begin{proof}
  By Lemma~\ref{l_lat}, $(\fM,\vee,\wedge)$ is a lattice. We first show that this lattice is distributive.
Let $\S=\S(\kappa)$ where $\kappa \subseteq X\times X$ is generated by rectangles.
Using Lemma~\ref{l_lat} and Theorem~\ref{th_oc}, we have
\begin{align*}
(\S_1\vee\S_2)\wedge \S
& =
\S(\oclint((\kappa_1\cup\kappa_2)\cap \kappa))
=
\S(\oclint((\kappa_1\cap \kappa)\cup (\kappa_2\cap \kappa)))\\
& =
\S(\oclint(\kappa_1\cap\kappa)\cup\oclint(\kappa_2\cap\kappa))
  =
(\S_1\wedge \S)\vee (\S_2\wedge \S).
\end{align*}
In addition, since $\kappa\cong\oclint(\kappa)$, using Lemma \ref{l_lat}
and Theorem \ref{th_oc}, we have
\begin{align*}
  (\S_1\wedge\S_2)\vee \S
  &=
    \S(\oclint(\kappa_1\cap\kappa_2)\cup \kappa)
  =
  \S(\oclint((\kappa_1\cap\kappa_2)\cup\kappa))
  \\&=
  \S(\oclint((\kappa_1\cup\kappa)\cap(\kappa_2\cup\kappa))
  =
  (\S_1\vee\S)\wedge(\S_2\vee \S).
\end{align*}

We now check the remaining axioms for a Boolean algebra:
\begin{itemize}
\item Since $\mathbf{0}=\S(\emptyset)$ and $\mathbf{1}=\S(X\times X)$, it is immediate from Lemma~\ref{l_lat} that $\S\wedge \mathbf{1}=\S=\S\vee\mathbf{0}$.
\item We have $\kappa\cong\oclint(\kappa)$, so, using (\ref{eq_com'}),
  \[ (\kappa')'=\ocl((\kappa')^c)=\ocl(\ointer(\kappa))=\oclint(\kappa)\cong \kappa.\]
  Hence, $(\S')'=\S$.
\item We have $\S\wedge\S' = \S(\oclint(\kappa \cap \ocl(\kappa^c)))$,
and $\kappa\cap\ocl(\kappa^c)=\kappa\setminus\ointer(\kappa)$ has empty $\omega$-interior
(see (\ref{eq_com'})). Hence, $\S\wedge\S'=\mathbf{0}$.
\item We have $\S\vee\S' = \S(\kappa\cup \ocl(\kappa^c))$ and $\kappa\cup\ocl(\kappa^c)\supseteq \kappa\cup\kappa=X\times X$. Hence, by Lemma \ref{l_lat}, $\S\vee\S'=\mathbf{1}$.\qedhere
\end{itemize}
\end{proof}

We now introduce some subclasses of $\fM$ and, in the remainder of
the section, examine their lattice-theoretic properties. Let

\begin{enumerate}[label=(\alph*)]
\item $\widehat \fS$ be the set of all strongly reflexive self-adjoint $\D$-bimodules in $\B(H)$;
\item $\fS$ be the set of all strongly reflexive operator $\D$-systems in $\B(H)$;
\item $\widehat\fF$ be the set of all strongly reflexive positive-full self-adjoint $\D$-bimodules in $\B(H)$;
\item $\fF$ be the set of all strongly reflexive positive-full operator $\D$-systems in $\B(H)$;
\item $\fN$ be the set of all strongly reflexive positive-null operator $\D$-systems in $\B(H)$.
\end{enumerate}

We then have that
\[\fS\subseteq \widehat\fS,\qquad \fF\subseteq \widehat\fF,\qquad\fF=\widehat\fF\cap\fS.\]
The Hasse diagram of inclusions for these sets of strongly reflexive masa-bimodules is as follows:
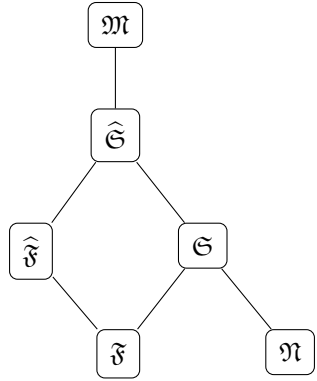
\[\begin{tikzpicture}[
  node distance=.8cm and .5cm,
  every node/.style={
    draw, rounded corners=3pt,
    inner sep=5pt,
    font=\small
  },
  >=stealth,
  edge/.style={draw, -}
]
% --- not neccesarily unital layer (left) ---
\node (M)                              {${\mathfrak{M}}$};
\node (hS) [below=of M]               {$\widehat{\mathfrak{S}}$};
\node (hF) [below left=of hS]          {$\widehat{\mathfrak{F}}$};
% --- unital layer (right, offset diagonally) ---
\node (S)  [below right=of hS]         {$\mathfrak{S}$};
\node (F)  [below left=of S]           {$\mathfrak{F}$};
\node (N)  [below right=of S]          {$\mathfrak{N}$};
% --- vertical edges within hat layer ---
\draw[edge] (M) -- (hS);
\draw[edge] (hS) -- (hF);
% --- vertical edges within plain layer ---
\draw[edge] (S)  -- (F);
\draw[edge] (S)  -- (N);
% --- cross edges (plain \subseteq hat) ---
\draw[edge] (S)  -- (hS);
\draw[edge] (F)  -- (hF);
\end{tikzpicture}
\]

\begin{lemma}\label{lem_c}
  Let $\kappa\subseteq X\times X$ be generated by
  rectangles. The following are equivalent:
  \begin{enumerate}
  \item $\osq(\kappa)$ is marginally null;
  \item $\Delta\cap\ointer(\kappa)$ is marginally null;
  \item $\Delta\oincl\kappa'$.
  \end{enumerate}
\end{lemma}

\begin{proof}
(ii)$\Leftrightarrow$(iii) is immediate from (\ref{eq_com'}).

(iii)$\Rightarrow$(i). Observe first that $\osq(\Delta^c)$ is marginally null, because any non-trivial square $\alpha\times\alpha$ intersects $\Delta$ in $\{(x,x):x\in \alpha\}\not\cong \emptyset$. Hence, if $\Delta\oincl\kappa'$, then
  \[\osq(\kappa) \oincl \ointer(\kappa)=(\kappa')^c\oincl \Delta^c,\]
  so $\osq(\kappa)\oincl\osq(\Delta^c)\cong\emptyset$, and (i) follows.

(i)$\Rightarrow$(ii). If (ii) is false, then there exists a non-trivial rectangle $R=\alpha\times \beta\oincl \ointer(\kappa)$ such that $\Delta\cap R\not\cong \emptyset$. This implies that $\gamma=\alpha\cap\beta\subseteq X$ is not null, and the non-trivial square $\gamma\times\gamma$ is contained in $\kappa$, which implies $\osq(\kappa)$ is not marginally null.
\end{proof}

\begin{lemma}\label{l_od}
  Let $\kappa\subseteq X\times X$ be a measurable set. If $\Delta\oincl\oclsq(\kappa)$, then
  $\Delta\oincl\osq(\kappa)$.
\end{lemma}

\begin{proof}
Write $\osq(\kappa) \cong \bigcup_{i\in \bN}\alpha_i\times \alpha_i$ for measurable sets $\alpha_i\subseteq X$, and
let $M = \left(\bigcup_{i\in \bN}\alpha_i\right)^c$. Trivially,
$\osq(\kappa)\subseteq (M\times M)^c$ and since $(M\times M)^c$ is $\omega$-closed,
we have that $\oclsq(\kappa)\subseteq_{\omega} (M\times M)^c$, that is,
$(M\times M)\cap \oclsq(\kappa) \cong \emptyset$.
The condition $\Delta\oincl\oclsq(\kappa)$ implies that
$(M\times M)\cap \Delta \cong \emptyset$ which shows that $M$ is null, and hence
$\Delta\oincl\osq(\kappa)$.
\end{proof}

\noindent{\bf Remark. }
The proof of Lemma \ref{l_od} shows that the classes $\widehat\fF$ and $\fF$ differ
in a rather superficial way: if a masa-bimodule $\cl U$ in $\widehat\fF$ is non-degenerate
in the sense that $\overline{[\cl UH]} = H$ then $\cl U$ is in fact in $\fF$.

\begin{lemma}\label{l_f}
For any $\S_1,\S_2\in \fM$, we have $\S_1\wedge_+\S_2:=\S(\oclsq(\kappa_1\cap\kappa_2))\in \widehat{\fF}$. Moreover,
\begin{enumerate}
\item[(i)]
$\S_1\wedge_+\S_2$ is the greatest lower bound of $\S_1$ and $\S_2$ in $\widehat{\fF}$ with respect to inclusion, and
\item[(ii)] if $\S_1,\S_2\in \fF$, then $\S_1\wedge_+\S_2\in \fF$.
\end{enumerate}
\end{lemma}

\begin{proof}
Apart from assertion (ii), the arguments are identical to the corresponding arguments for $\wedge$ in Lemma~\ref{l_lat}: we simply substitute square $\omega$-interiors, Theorem~\ref{th_psr} and Proposition~\ref{l_goi} (ii), where in the previous proof we used $\omega$-interiors, Theorem~\ref{th_sr} and Proposition~\ref{l_goi} (i).

To show (ii): suppose that $\S_1,\S_2\in \fF$. Then $\Delta\oincl\kappa_i\cong\oclsq(\kappa_i)$ for $i=1,2$, where $\kappa_i=\osupp(\S_i)$. By Lemma~\ref{l_od}, we have $\Delta\oincl\osq(\kappa_i)$ for $i=1,2$; hence, using Lemma \ref{lem_ointer},
$$\Delta\oincl\osq(\kappa_1)\cap\osq(\kappa_2)\cong \osq(\kappa_1\cap\kappa_2).$$
It follows that $\Delta\oincl\oclsq(\kappa_1\cap\kappa_2)=\osupp(\S_1\wedge_+\S_2)$, so
$\S_1\wedge_+\S_2\in \fF$.
\end{proof}

We then have the lattice-theoretic properties for some subclasses of $\fM$.

\begin{theorem}\label{th_com_rest}
The following hold:
\begin{enumerate}
\item the set $\widehat\fS$ is a Boolean subalgebra of $\fM$;
\item $(\widehat\fF,\vee,\wedge_+)$ is a lattice which contains $\fF$ as a sublattice;
\item the set $\fN$ is closed under complementation.
\end{enumerate}
\end{theorem}

\begin{proof}
Let $\S_1,\S_2\in \fM$. By Theorem~\ref{th_sr}, we have $\S_i = \S(\kappa_i)$, for $i=1,2$, where
$\kappa_i\subseteq X\times X$ is generated by rectangles.

(i)
Since the adjoint operation is weakly continuous, the subspaces
$\S_1\vee \S_2$ and $\S_1\wedge \S_2$ are self-adjoint in view of their
definition (\ref{eq_defofvee}).
It suffices to show that $\widehat\fS$ is closed under complementation. Suppose that
$\cl S\in \widehat\fS$ and let
$\kappa\subseteq X\times X$ be an $\omega$-closed and symmetric set, generated by rectangles,
such that $\cl S = \cl S(\kappa)$.
If $R\subseteq_{\omega} \kappa$ is a rectangle, the symmetry of $\kappa$ implies that
$R^t \subseteq_{\omega} \kappa$. Thus, ${\rm int}_{\omega}(\kappa)$ is symmetric and hence
so is ${\rm int}_{\omega}(\kappa)^c$. By (\ref{eq_com'}), $\kappa'$ is symmetric. Thus,
$\cl S' = \cl S(\kappa')\in \widehat\fS$.

(ii)
Suppose that $\S_1,\S_2\in \fF$.
In view of Theorem \ref{th_psr} and Remark \ref{r_gsq}, let $\Omega_i$ be a square $\omega$-open set such that
$\kappa_i = \ocl(\Omega_i)$, $i = 1,2$.
Clearly, $\Omega_1 \cup \Omega_2$ is square $\omega$-open.
By Corollary \ref{lem_oc}\,(iv),
$\kappa_1\cup \kappa_2 \cong \ocl(\Omega_1\cup \Omega_2)$.
By Lemma~\ref{l_lat}, we have $\S_1 \vee \S_2 \in \fF$.
On the other hand, by Lemma~\ref{l_f}, $\S=\S_1 \wedge_+ \S_2 \in \fF$ and $\S$ is the greatest lower bound of $\S_1,\S_2$ in $\fF$. Identical arguments apply to $\S_1,\S_2\in\widehat\fF$.

(iii)
If $\kappa=\osupp(\S)$ where $\S\in \fN$, then $\Delta\oincl\kappa$ and $\osq(\kappa)$ is marginally null.
By (\ref{eq_com'}), $(\kappa')'\cong\kappa$ and hence
Lemma~\ref{lem_c} implies that $\Delta\oincl\kappa'$ and
$\osq(\kappa')$ is marginally null. Thus, $\S'=\S(\kappa')\in \fN$.

\end{proof}

\begin{remark*}
  Note that it can happen that $\kappa_1$ and $\kappa_2$ are both generated by rectangles
  but $\kappa_1\cap \kappa_2$ is not. For example, let
  $X = [0,1]$, equipped with Lebesgue measure, and let
  $\kappa_1 = \{(x,y) : x\leq y\}$ and $\kappa_2 = \{(x,y) : x\geq y\}$.
  Then $\kappa_1 \cap \kappa_2 = \{(x,x) : x\in X\}$ and its $\omega$-interior is empty, although it is not
  marginally equivalent to the empty set.
\end{remark*}

We do not know if
the operations $\wedge_+$ and $\wedge$ coincide.
The question is related to the following problem: Suppose that $\Omega$ is an $\omega$-open set and
$\kappa = \ocl(\Omega)$. Under what conditions is it true that $\Omega = \ointer(\kappa)$?

\section{Masa-bimodules of Toeplitz type}\label{s_mbtt}

We now specialise our framework to the setting of locally compact groups, focusing on operator systems of Toeplitz type, which arise from closed subsets of the group.
By translating our lattice families into this context, we demonstrate that membership in these spaces is characterised by topological properties of the underlying subset, such as symmetry and the location of the neutral element.

Fix a locally compact second countable group $G$ with neutral element $e$,
and let $X = G$, equipped with Haar measure $\mu$.
Let $\rho$ be any metric on $G$ that induces its topology, and write $B(t,\epsilon)$ for the open ball with centre $t\in G$
and radius $\epsilon > 0$.
If $E\subseteq G$ is a Borel set, we say that $t\in G$ is a \emph{point of density} for $E$ if
$\lim_{\epsilon\to 0} \frac{\mu(E\cap B(t,\epsilon))}{\mu(B(t,\epsilon))} = 1$.
We say that $G$ \emph{satisfies the Lebesgue Density Theorem} if, for every measurable subset $E\subseteq G$,
the set
$$\{t\in E : t \text{ is not a point of density for } E\}$$
is null.
It is well-known that $\bb{R}^n$, and hence any Lie group,
satisfies the Lebesgue Density Theorem.

If $E\subseteq G$, let
$$E^* = \{(s,t)\in G\times G : ts^{-1} \in E\}.$$
We denote by ${\rm int}(E)$ the topological interior of a subset $E\subseteq G$, and by
${\rm cl}(E)$ its topological closure. We let, as usual, $\partial(E) = {\rm cl}(E) \setminus {\rm int}(E)$ be the
boundary of $E$.

We recall the following result from \cite{llt}:

\begin{theorem}[{\cite[Theorem 5.2]{llt}}]\label{th_into}
Let $G$ be a locally compact group satisfying the Lebesgue Density
  Theorem and let $E\subseteq G$ be a Borel subset of positive measure. Then
\begin{equation}\label{eq_clinte}
  \ointer(E^*) \cong {\rm int}(E)^*\quad\text{and}\quad \ocl(E^*) \cong {\rm cl}(E)^*.
\end{equation}
  In particular,
  \begin{enumerate}
\item $E^*$ is a positivity domain if and only if $E$ is a symmetric set, $e\in E$, and
  $E$ is the closure of its interior;

\item $E^*$ is generated by squares
if and only if $E^*$ contains a non-trivial square,
if and only if $E$ is a symmetric set, $e\in E$,  $E$ is the closure of its interior
  and $E$ contains a symmetric open neighbourhood of $e$.
\end{enumerate}
\end{theorem}

We note that the masa-bimodules of the form $\S(E^*)$, where $E\subseteq G$ is a closed set,
called masa-bimodules \emph{of Toeplitz type}, were studied in \cite{akt, akt2, lt}, among
others.

\begin{theorem}\label{th_toep}
Let $G$ be a locally compact group satisfying the Lebesgue Density Theorem and
let $E\subseteq G$ be a closed set. Then
$\S(E^*) \in \fM$ if and only if $E$ is the closure of its interior.
Moreover,
\begin{enumerate}
\item $\S(E^*)  \in \fS$ if and only if $E$ is a symmetric set, $e\in E$ and
$E$ is the closure of its interior;
\item $\S(E^*) \in \fF$ if and only if $E$ is a symmetric set, $e\in {\rm int}(E)$ and
$E$ is the closure of its interior;
\item $\S(E^*)\in \fN$ if and only if $E$ is a symmetric set, $e\in \partial(E)$ and $E$ is
the closure of its interior.
\end{enumerate}
In addition,
\begin{enumerate}[label={\rm (\alph*)}]
\item if $E_1$ and $E_2$ are subsets of $G$ that are the closures of their interiors, then
$$\S(E_1^*) \vee \S(E_2^*) = \S((E_1\cup E_2)^*) \text{ and } \S(E_1^*) \wedge \S(E_2^*)
= \S({\rm cl}({\rm int}(E_1\cap E_2))^*);$$
\item if $e\in \partial(E)$, then
$\S(E^*)' = \S({\rm cl}(E^c)^*);$
\item the class $\fF_{\rm Toep}$ of all masa-bimodules of Toeplitz type in $\fF$ is a sublattice of $\fS$.
\end{enumerate}
\end{theorem}

\begin{proof}
The first statement and (i) are immediate consequences of Theorem \ref{th_into}.

(ii) follows from Theorems \ref{th_psr} and \ref{th_into}.

(iii) If $e\in {\rm int}(E)$ then, because $E$ is symmetric, $E$ contains a symmetric open set $U$ with $e\in U$, and hence by
Theorem \ref{th_into}, $\S(E^*)$ is positive-full. Thus, if $\S(E^*)$ is positive-null then $e\in E\setminus {\rm int}(E)$,
that is, $e\in \partial(E)$. Conversely, if $\S(E^*)$ is not positive-null then $\S(E^*)$
contains a positive rank one operator, and therefore $E^*$ contains a non-trivial square $\alpha\times\alpha$.
The claim follows from Theorem \ref{th_into}.

(a) is immediate from Lemma \ref{l_lat} and identities (\ref{eq_clinte}).

(b) is immediate from the definition of complementary sets and identities (\ref{eq_clinte}).

(c)
It follows from Theorem \ref{th_into}\,(ii) that if $\S(E_1^*), \S(E_2^*) \in \fF$ then
${\rm int}(E_1\cap E_2)$ is symmetric and contains $e$. By Theorem \ref{th_into}\,(ii) again,
${\rm cl}({\rm int}(E_1\cap E_2))^*$ is generated by squares; by Theorem \ref{th_psr},
$\S(E_1^*) \wedge \S(E_2^*)$ is positive-full.
Since $\S(E_1^*) \wedge_+ \S(E_2^*)$ is the largest positive-full operator $\D$-system
contained in $\S(E_1^*) \cap \S(E_2^*)$, we have that
$\S(E_1^*) \wedge \S(E_2^*) = \S(E_1^*) \wedge_+ \S(E_2^*)$.
The claim now follows from Theorem \ref{th_com_rest}\,(ii).
\end{proof}

\begin{remark*}
  Let $G$ be a locally compact group satisfying the Lebesgue Density Theorem, and let
  $V\subseteq G$ be an open symmetric dense set containing $e$, where $V\neq G$.
  By Theorem \ref{th_into}, $\ocl(V^*) \cong G\times G$, and thus
  $\ointer(\ocl(V^*))\neq V^*$. In addition, $V^*$ is $\omega$-dense but distinct, with respect to
  marginal equivalence, from $G\times G$.
\end{remark*}

We finish with an example that relates to the lattice structures examined in Section \ref{s_lmb}.

\begin{example}
We give an example showing that $\fN$ is not closed under $\vee$ or $\wedge$,
which also shows that $\fS$ is not closed under $\wedge$.
Let $X=[0,1]$, equipped with Lebesgue measure. For
  $\gamma\subseteq X$, let
  $\theta(\gamma)=\{(x,y)\in X \times X : |x-y|\in \gamma\}$. Note that sets of this form are symmetric, they contain $\Delta$ provided $0\in \gamma$, and
  $\ocl(\theta(\gamma))=\theta(\overline{\gamma})$ and
  $\ointer(\theta(\gamma))=\theta({\rm int}(\gamma))$. In particular, if
  $\gamma$ is closed and ${\rm int}(\gamma)$ is dense in $\gamma$, then $\theta(\gamma)$ is
  generated by rectangles. Moreover,
  $\osq(\theta(\gamma))\not\cong \emptyset$ if and only if $\gamma$
  contains a neighbourhood of $0$.

 Let
  \[\gamma_1=\{0\}\cup\bigcup_{k\ge0}[2^{-(2k+1)},2^{-2k}] \ \mbox{ and } \
    \gamma_2=\{0\}\cup \bigcup_{k\ge0}[2^{-(2k+2)},2^{-(2k+1)}],\]
viewed as subsets of $X=[0,1]$, and $\kappa_i= \theta(\gamma_i)$, $i = 1,2$.
For $i=1,2$, the subspace
  $\S_i=\S(\kappa_i)$ is generated by rectangles and has empty square
  $\omega$-interior, thus $\S_i\in \fN$, by Theorems~\ref{th_into} and~\ref{th_toep}.
  But $\gamma_1\cup
  \gamma_2 = X$, so $\kappa_1\cup \kappa_2=X \times X$, thus, $\S_1\vee\S_2=\S(\kappa_1\cup\kappa_2) = \cl B(H)\not\in \fN$.

Observe that $\gamma_1\cap\gamma_2$ has empty interior. It follows that $\S_1\wedge\S_2=\{0\}$, which is not in $\fS$, and hence is not in $\fN$. Therefore neither $\fN$ nor $\fS$ is closed under $\wedge$.
\end{example}

\end{document}